\documentclass{amsart}

\usepackage{amsmath, amsthm, amssymb, amsfonts, amscd, latexsym, graphics}
\usepackage{mathrsfs}
\usepackage{enumerate}
\usepackage{geometry}
\usepackage{amsfonts}

\makeatletter
\@namedef{subjclassname@2020}{%
  \textup{2020} Mathematics Subject Classification}
\makeatother

\newcommand{\bbP}{\mathbb{P}}
\newcommand{\bbZ}{\mathbb{Z}}
\newcommand{\bbN}{\mathbb{N}}

\newcommand{\bfg}{\mathbf{g}}
\newcommand{\bfone}{\mathbf{1}}

\newcommand{\Cl}{\mathrm{Cl}}

\newcommand{\lGrMod}{\mathchar`-\mathsf{GrMod}}
\newcommand{\lFilt}{\mathchar`-\mathsf{Filt}}

\DeclareMathOperator{\dimk}{dim} 

\DeclareMathSymbol{\Xdsum}{\mathop}{largesymbols}{88}
\DeclareMathSymbol{\Xtsum}{\mathop}{largesymbols}{80}
\DeclareMathOperator*{\dsum}{\mathchoice{\Xdsum}{\Xdsum}{\Xtsum}{\Xtsum}}

\DeclareMathSymbol{\Xdbigoplus}{\mathop}{largesymbols}{77}
\DeclareMathSymbol{\Xtbigoplus}{\mathop}{largesymbols}{76}
\DeclareMathOperator*{\dbigoplus}{\mathchoice{\Xdbigoplus}{\Xdbigoplus}{\Xtbigoplus}{\Xtbigoplus}}
\newcommand\Dbigoplus{\dbigoplus\limits}

\usepackage{color}

\newcommand{\blue}[1]{#1}
\newtheorem{thm}{Theorem}[section]
\newtheorem{cor}[thm]{Corollary}
\newtheorem{lem}[thm]{Lemma}
\newtheorem{prop}[thm]{Proposition}

\theoremstyle{definition}
\newtheorem{dfn}[thm]{Definition}
\newtheorem*{dfn*}{Definition}
\newtheorem{ex}[thm]{Example}

\theoremstyle{remark}
\newtheorem{rem}[thm]{Remark}

\newcommand{\thmref}[1]{Theorem~\ref{#1}}
\newcommand{\lemref}[1]{Lemma~\ref{#1}}
\newcommand{\corref}[1]{Corollary~\ref{#1}}
\newcommand{\propref}[1]{Proposition~\ref{#1}}

\newcommand{\dfnref}[1]{Definition~\ref{#1}}
\newcommand{\remref}[1]{Remark~\ref{#1}}
\newcommand{\secref}[1]{Section~\ref{#1}}

\numberwithin{equation}{section}

\title[Point modules over the universal enveloping alg. of color Lie alg.]{Point modules over the universal enveloping algebras of color Lie algebras}

\author{Shu Minaki}
\address{
Department of Mathematics, 
Graduate School of Science, 
Tokyo University of Science,
1-3 Kagurazaka, Shinjuku-ku, Tokyo, 162-8601, JAPAN}
\email{1125704@ed.tus.ac.jp} 

\begin{document}

\begin{abstract}
Let $k$ be an algebraically closed field with characteristic zero.
In this paper, we define the notion of a $q'$-Heisenberg normal element of a $\mathbb{Z}$-graded $k$-algebra.
This $q'$-Heisenberg normal element gives the structure of some sets of modules related to point modules.
We also determine the set of point modules over an Artin--Schelter regular algebra obtained as the universal enveloping algebra of a color Lie algebra.
Moreover, we give a concrete integer such that the inverse system of its truncated point schemes is constant.
This is a quantitative answer to a question raised by Artin--Tate--Van den Bergh, in our setting.
\end{abstract}
\maketitle
\section{Introduction}
Throughout this paper, let $k$ be an algebraically closed field with characteristic zero.
For a $k$-algebra $A$, an $A$-module refers to a left $A$-module.
For a graded algebra $A$, $A\lGrMod$ denotes the (left) graded module category.
For a (left) graded $A$-module $M$ and $m \in \bbZ$, $M(m)$ denotes the shift $M(m)_{i}=M_{i+m}$ for $i \in \bbZ$.
Moreover, for a graded $k$-vector space $V=\bigoplus_{i \in I} V_{i}$ where $I \neq \emptyset$, $h(V)=\bigcup_{i \in I} V_{i}$ denotes the set of homogeneous elements of $V$, and a map 
$|\quad|:h(V)\setminus\{0\} \rightarrow I$ is defined as $x \mapsto i \text{ such that } x\in V_{i}$.

An Artin--Schelter regular algebra (or AS-regular algebra, in short) was introduced by Artin--Schelter \cite{AS}.
By the fact that a commutative AS-regular algebra is a polynomial algebra, an AS-regular algebra is regarded as a ``non-commutative polynomial algebra''.
Throughout this paper, we assume that an AS-regular algebra is generated in degree $1$.
AS-regular algebras are often studied by a geometric method.
One geometric method is to study a point scheme introduced by Artin--Tate--Van den Bergh \cite{ATV90}\textup{:}
Let $A=k\langle x_{0},x_{1},\ldots, x_{s} \rangle/\langle f_{1}, \ldots, f_{t} \rangle$ be a finitely presented connected $\bbN$-graded $k$-algebra generated in degree $1$, $V$ the $k$-vector space $\langle x_{0},x_{1},\ldots, x_{s} \rangle_{k}$, and $I$ the homogeneous two-sided ideal $\langle f_{1}, \ldots, f_{t}\rangle $ of $k\langle V \rangle$.
  For each integer $i \geq 1$, 
  an element of $I_{i}$ defines a multi-linear function on $(V^{*})^{\times i}$.
  The \textit{$i$-th truncated point scheme} $\Gamma_{i}(A) \subset (\mathbb{P}^{s})^{\times i}$ of $A$ is defined as the zero scheme defined by $I_{i}$.
  Moreover, $\mathrm{pr}^{(i+1)}_{1,i}:\,\Gamma_{i+1}(A) \rightarrow \Gamma_{i}(A)$ denotes the restriction of the projection $(\mathbb{P}^{s})^{\times i+1} \rightarrow (\mathbb{P}^{s})^{\times i}$ onto the first $i$ factors.
   The \textit{point scheme} $\Gamma(A)$ of $A$ is defined as the inverse limit of $\Gamma_{i}(A)$\textup{:} $\Gamma(A):=\displaystyle{\lim_{\leftarrow}\Gamma_{i}(A)}$ (\cite{ATV90}).

For a three-dimensional AS-regular algebra $A$, the point scheme $\Gamma(A)$ of $A$ is isomorphic to a product of the projective space either $\bbP^{2}$ or $\bbP^{1}\times \bbP^{1}$, or a divisor of either.
In \cite{ATV90}, three-dimensional AS-regular algebras were classified by point schemes.
On the other hand, for a four or higher dimensional AS-regular algebra, the dimension of the (truncated) point scheme may be low.
For four-dimensional AS-regular algebras with three or higher Gelfand--Kirillov dimensions, 
Lu--Palmieri--Wu--Zhang \cite{LPWZ} listed the numbers of generators and relations with degrees.
Let $R$ be a set of relations that have degrees on this list; then, the dimension of a truncated point scheme of $R$ is expected to be zero from Chirvasitu--Kanda \cite{CK}.
Moreover, by Vancliff--Van Rompay--Willaert \cite{VVW98} and Vancliff--Veerapen \cite{VV14}, some concrete examples were constructed in which the four-dimensional AS-regular algebras have the zero-dimensional point scheme.
Moreover, Vancliff \cite{Van} constructed a five-dimensional AS-regular algebra whose point scheme has no closed points.
On the other hand, 
the point variety of a skew polynomial algebra, which is a quadratic AS-regular algebra, can have the ``large'' point scheme (Belmans--De Laet--Le Bruyn \cite{BDL} and  Vitoria \cite{Vit}). 
In this paper, by using color Lie algebras, we construct AS-regular algebras of any dimension with the ``large'' point scheme and relations having some degrees.

A color Lie algebra was introduced by Ree \cite{Ree} as a generalized Lie algebra.
As a non-commutative algebro-geometric study at an early stage, Van Oystaeyen--Willaert studied it as a schematic algebra \cite{VW}.
Moreover, for a color Lie algebra $L$, the universal enveloping algebra of $L$ is an AS-regular algebra under some conditions.

Let $A$ be a finitely presented connected $\bbN$-graded $k$-algebra $A$ generated in degree $1$.
 A point module over $A$ is important by the following fact (\cite{ATV90})\textup{:} 
 A one-to-one correspondence exists between the set of isomorphism classes of point modules over $A$ and the set of closed points of the point scheme of $A$.
In this paper, to study a point module, we define the notion of a \textit{$q'$-Heisenberg normal element} of a $\bbZ$-graded $k$-algebra.
\begin{dfn*}[{\textrm{\textbf{\dfnref{dfn-one}}}}]
Let $A$ be a $\bbZ$-graded $k$-algebra.
A regular homogeneous normal element $g$ of $A$ with degree $n \geq 1$ is called a \textit{$q'$-Heisenberg normal element} of $A$ if there exist $u \in k^{\times}$, $x \in A_{1}$, and $y \in A_{n-1}$ that satisfy the following equalities\textup{:}
\begin{enumerate}
  \item[(i)] $g=xy-uyx$;
  \item[(ii)] $0=xg-ugx$;
  \item[(iii)] $0=gy-uyg$.
\end{enumerate}
\end{dfn*}
The relations in the above definition of a $q'$-Heisenberg normal element are similar to defining relations of another $q$-analog of the universal enveloping algebra of the Heisenberg algebra (cf. Kirkman--Small \cite{KS}).
To emphasize this, we use the symbol $q'$ instead of merely using $q$.
This $q'$-Heisenberg normal element gives us the structure of some sets of modules related to a point module.
Assume that $A$ is a finitely presented connected $\bbN$-graded $k$-algebra generated in degree $1$ and $g$ is a $q'$-Heisenberg normal element of $A$.
\propref{prop-fat-point} shows that the set of base-point modules over $A$ completely overlaps one of $A/\langle g \rangle$.
Also, in \thmref{thm-tru-pmod}, we obtain a criterion for the nonexistence of a truncated point module.
Therefore, we obtain the fact that some set of truncated point modules over $A$ completely overlaps one of $A/\langle g \rangle$.
In \thmref{thm-color-var}, using this result, we determine the structure of the set of closed points of the point scheme of the universal enveloping algebra of color Lie algebras.

For a finitely presented connected $\bbN$-graded $k$-algebra $A$, Artin--Tate--Van den Bergh raised the following question \cite[Question 3.16]{ATV90}\textup{:}
When $A$ is noetherian, is the inverse system $\{\Gamma_{i}(A)\}$ constant for large $i$?
A partial answer to this question was given by Artin--Zhang \cite{AZ01}\textup{:} 
When $A$ is strongly noetherian, $\{\Gamma_{i}(A)\}$ is constant for some $i >0 $.
In \corref{cor-color-sch}, for the universal enveloping algebra $U(L)$ of a color Lie algebra $L$, we give a concrete integer $i$ such that $\{\Gamma_{i}(U(L))\}$ is constant.
Note that this integer $i$ is obtained by a property of the color Lie algebra.

The remainder of this paper is organized as follows\textup{:}
 In \secref{sec-pre}, we recall some definitions and properties of the quasi-Veronese algebra defined by Mori \cite{Mori}; the twisted algebra defined by Zhang \cite{Zhang}; dehomogenizations, point modules, and color Lie algebras.
 In \secref{sec-gen}, we define the notion of a $q'$-Heisenberg normal element and prove basic properties.
In \propref{prop-fat-point}, we prove the nonexistence of a $g$-torsionfree base-point module for a $q'$-Heisenberg normal element $g$.
In \secref{sec-trunc}, we prove \thmref{thm-tru-pmod}, which is a theorem of the nonexistence of a $g$-torsionfree module with some conditions.
By \thmref{thm-tru-pmod}, we determine the set of closed points of the point schemes of the universal enveloping algebra of color Lie algebras in \thmref{thm-color-var}.
At the end of this section, we show that the inverse system of the truncated point schemes of it is constant in \corref{cor-color-sch}.
\section{Preliminaries}\label{sec-pre}
\subsection{Graded and filtered algebras}\label{subsec-gr-filt}
In this subsection, we recall the categorical equivalences obtained by quasi-Veronese algebras \cite{Mori}, twisted algebras \cite{Zhang}, and dehomogenizations (see \cite{HvO}).
First, we recall the definition of the quasi-Veronese algebra.
\begin{dfn}[{\cite[Definition 3.7]{Mori}}]
  Let $A=\bigoplus_{i \in \bbZ}A_{i}$ be a $\bbZ$-graded $k$-algebra and $r \geq 1$ a positive integer.
  A $\bbZ$-graded $k$-algebra $A^{[r]}$ is called the \textit{$r$-th quasi-Veronese algebra} of $A$ defined as follows\textup{:}
  \[
  A^{[r]}:=\bigoplus_{i \in \bbZ}
\left(\begin{array}{cccc}
A_{ri} & A_{ri+1} & \cdots & A_{ri+r-1}\\
A_{ri-1} & A_{ri}  & \cdots & A_{ri+r-2}\\ 
\vdots && \ddots & \vdots \\
A_{ri-r+1} & A_{ri-r+2} &  \cdots & A_{ri}\\ 
\end{array}\right)
  \]
  with the multiplication $(a_{i,j})(b_{i,j})=\left( \sum_{l=0}^{r-1}a_{l,j}b_{i,l} \right)$
  for $(a_{i,j}) \in (A^{[r]})_{p}$, $(b_{i,j}) \in (A^{[r]})_{q}$ for $p,q \in \bbZ$.
\end{dfn}
The definition of the quasi-Veronese algebra gives the following categorical equivalence\textup{:}
\begin{lem}[{\cite[Lemma 3.9]{Mori}}]\label{lem-quasi-Vero}
  For a positive integer $r \geq 1$, let $A=\bigoplus_{i \in \bbZ}A_{i}$ be a $\bbZ$-graded $k$-algebra and $A^{[r]}$ the $r$-th quasi-Veronese algebra of $A$.
  Then, the following categorical equivalence holds\textup{:} 
  \[
  A\lGrMod \cong A^{[r]}\lGrMod.
  \]
  This categorical equivalence is given by the functor $Q: A\lGrMod \rightarrow A^{[r]}\lGrMod$\textup{;}
  \[
  Q(M):=\bigoplus_{i \in \bbZ}\left(\bigoplus_{j=0}^{r-1}M_{ir-j}\right)
  \]
  with the action $(a_{i,j})(m_{i})=\left(\sum_{l=0}^{r-1}a_{l,i}m_{l}\right)$ for $(a_{i,j}) \in (A^{[r]})_{p}, (m_{i}) \in (Q(M))_{q}$ for $p,q \in \bbZ$.
\end{lem}
Next, we recall the definition of the twisted algebra by Zhang \cite{Zhang}.
Here, to deal with the left module, we recall the left version of it.
\begin{dfn}[{\cite[Definition 4.1]{Zhang}}]
Let $A=\bigoplus_{i \in \bbZ}A_{i}$ be a $\bbZ$-graded $k$-algebra.
A set of graded $k$-linear automorphisms $\nu=\{\nu_{i} \mid i \in \bbZ \}$ of $A$ is called an \textit{\textup{(}$\ell$-\textup{)}twisting system} if $\nu$ satisfies the following equality\textup{:}
\[
\nu_{l}(\nu_{j}(a)b)=\nu_{j+l}(a)\nu_{l}(b)
\,\text{ for } 
i,j,l \in \bbZ, a \in A_{i}, \text{ and } b \in A_{j}.
\]
\end{dfn}
\begin{dfn}[{\cite[Proposition and Definition 4.2]{Zhang}}]
  Let $A=\bigoplus_{i \in \bbZ}A_{i}$ be a $\bbZ$-graded $k$-algebra and $\nu=\{\nu_{i} \mid i \in \bbZ \}$ an ($\ell$-)twisting system of $A$.
  A $\bbZ$-graded $k$-algebra ${}^{\nu}A$ is called the \textit{\textup{(}$\ell$-\textup{)}twisted algebra} defined as follows\textup{:}

    ${}^{\nu}A:=\Dbigoplus_{i \in \bbZ}A_{i}$  
    with the multiplication $a\circ b:= \nu_{j}(a)b$, where $a \in {}^{\nu}A_{i}$, $b \in {}^{\nu}A_{j}$, and $i,j \in \bbZ$.
\end{dfn}
The definition of the ($\ell$-)twisted algebra gives the equivalence of the (left) graded module categories.
\begin{prop}[{\cite[Corollary 4.4]{Zhang}}]\label{prop-Zhang-twi}
  Let $A=\bigoplus_{i \in \bbZ}A_{i}$ be a $\bbZ$-graded $k$-algebra and $\nu$ a \textup{(}$\ell$-\textup{)}twisting system of $A$.
  Then, the following categorical equivalence holds\textup{:}
  \[
  A\lGrMod \cong {}^{\nu}A\lGrMod .
  \]
  This categorical equivalence is given by the functor ${}^{\nu}(-): A\lGrMod \rightarrow {}^{\nu}A\lGrMod $\textup{;}
  
  ${}^{\nu}M:=\bigoplus_{i \in \bbZ}M_{i}$  
     with the action $a\circ m:= \nu_{j}(a)m$, where $a \in {}^{\nu}A_{i}$, $m \in {}^{\nu}M_{j}$, and $i,j \in \bbZ$.
\end{prop}
We recall the definition and some facts of the dehomogenization (see \cite{HvO} for details).
\begin{dfn}[{see \cite[Chapter I. 4.3]{HvO}}]
Let $A=\bigoplus_{i \in \bbZ} A_{i}$ be a $\bbZ$-graded $k$-algebra, $M=\bigoplus_{i \in \bbZ}M_{i}$ a graded $A$-module, and $T$ a regular central element of $A$ with degree $1$.
A filtered ring $\mathcal{R}$ with the filtration $F\mathcal{R}=\{F_{i}\mathcal{R}\}_{i \in \bbZ}$ is called the \textit{dehomogenization} of $A$ with respect to $T$, defined as the following\textup{:}
\[
\mathcal{R}:=A/\langle 1-T \rangle,\quad
F_{i}\mathcal{R}:=(A_{i}+\langle 1-T \rangle)/\langle 1-T \rangle
\text{ for }
i \in \bbZ.
\]

Also, a filtered $\mathcal{R}$-module $\mathcal{M}$ with the filtration $F\mathcal{M}=\{F_{i}\mathcal{M}\}_{i \in \bbZ}$ is called the \textit{dehomogenization} of $M$ with respect to $T$, defined as the following\textup{:}
\[
\mathcal{M}:=M/\langle 1-T \rangle M,\quad
F_{i}\mathcal{M}:=(M_{i} + \langle 1-T \rangle M)/\langle 1-T \rangle M
\text{ for }
i \in \bbZ.
\]
\end{dfn}
The dehomogenization satisfies some useful properties, as stated below.
\begin{lem}[{see \cite[Chapter I 4.3 Lemma 2]{HvO}}]
Let $A=\bigoplus_{i \in \bbZ} A_{i}$ be a $\bbZ$-graded $k$-algebra, $M=\bigoplus_{i \in \bbZ}M_{i}$ a graded $A$-module, and $T$ a regular central element of $A$ with degree $1$.
Then, the equality $(1-T)A \cap A_{i}=0$ holds for $i \in \bbZ$.
Also, if $M$ is a $T$-torsionfree $A$-module, then the equality $(1-T)M \cap M_{i}=0$ holds for $i \in \bbZ$.
\end{lem}
\begin{prop}[{see \cite[Chapter I. 4.3.6]{HvO}}]
Let $A=\bigoplus_{i \in \bbZ} A_{i}$ be a $\bbZ$-graded $k$-algebra, $T$ a regular central element $T$ of $A$ with degree $1$, and $\mathcal{R}$ the dehomogenization of $A$ with respect to $T$.
Suppose that $\mathsf{F}_{T}$ is the full subcategory of $A\lGrMod$ consisting of $T$-torsionfree $A$-modules
and $\mathcal{R}\lFilt$ is the category of filtered $\mathcal{R}$-modules.
Then\textup{,} the dehomogenization gives the categorical equivalence as follows\textup{:}
  \[
  \mathcal{R}\lFilt \cong \mathsf{F}_{T}.
  \]
\end{prop}
\subsection{Point schemes and point modules}\label{subsec-point}
Let $A=k\langle  x_{0},x_{1}, \ldots, x_{s} \rangle/ \langle f_{1}, \ldots, f_{t} \rangle$ be a finitely presented connected $\bbN$-graded $k$-algebra generated in degree $1$ with generators $x_{i}$ with degree $1$ and relations $f_{j}$ with degree greater than or equal to $2$.
In this subsection, we recall the definition and the propositions of a point module over $A$. 
\begin{dfn}[{\cite{ATV90, OF}}]
Let $A$ be a finitely presented connected $\bbN$-graded $k$-algebra generated in degree $1$.
\begin{enumerate}
  \item
 A graded $A$-module $P=\bigoplus_{i \in \mathbb{Z}}P_{i}$ is called a \textit{point module} over $A$ if $P=AP_{0}$ and $P$ satisfies the equality
  $\dimk_{k}P_{i}=
  \begin{cases}
    1 &\text{for $i \geq 0$},\\
    0 &\text{for $i <0$}.
  \end{cases}$
  \item
For each integer $d \geq 1$, a graded $A$-module $P=\bigoplus_{i \in \mathbb{Z}}P_{i}$ is called a \textit{truncated point module of length $d+1$} over $A$ if $P=AP_{0}$ and $P$ satisfies the equality
  $\dimk_{k}P_{i}=
  \begin{cases}
    1 &\text{for $0 \leq i \leq d$},\\
    0 &\text{otherwise}.
   \end{cases}$
\end{enumerate}
\end{dfn}
Moreover, we recall the definition of a base-point module.
\begin{dfn}[{\cite{cCV}}]
 Let $A$ be a finitely presented connected $\bbN$-graded $k$-algebra generated in degree $1$ and $c \geq 1$ an integer. A graded $A$-module $M=\bigoplus_{i \in \mathbb{Z}}M_{i}$ is called a \textit{base-point module} over $A$ if $M=AM_{0}$, $M$ satisfies the following equality\textup{:}
\[
\dimk_{k}M_{i}=
  \begin{cases}
    c &\text{for $i \geq 0$};\\
    0 &\text{for $i < 0$};
   \end{cases}
\]
  and $M$ is $1$-critical with respect to the Gelfand--Kirillov dimension.
\end{dfn}
We remark, for a base-point module $M$, when $c=1$, $M$ is a point module. 
Also, when $c \geq 2$, $M$ is called a \textit{fat point module} (\cite{Ar90}).

For a $k$-algebra $A$, an element $g$ of $A$ satisfying an equality $gA=Ag$ is called a \textit{normal element}.
Also, for an $A$-module $M$, $M$ is called a \textit{$g$-torsionfree module} if $gm=0$ implies $m=0$ for $m \in M$.
To recall the well-known property for a point module and a normal element, we show a useful property as follows\textup{:}
\begin{lem}\label{lem-trunc-basic}
  Let $A$ be a finitely presented connected $\bbN$-graded $k$-algebra generated in degree $1$ 
  and $P=\bigoplus_{i=0}^{d}k m_{i}$ a truncated point module of length $d+1$ over $A$ \textup{(}resp. $P'=\bigoplus_{j \in \bbZ}k m'_{j}$ is a point module over $A$\textup{)}.
  Suppose that $g$ is a homogeneous normal element of $A$ with degree $n$.
  Then, $g m_{i-1}=0$ holds if and only if $g m_{i}=0$ holds for $1 \leq i \leq d-n$ \textup{(}resp. $g m'_{j-1}=0$ holds if and only if $g m'_{j}=0$ holds for $j \geq 1$\textup{)}.
\end{lem}
\begin{proof}
  Since a truncated point module $P$ over $A$ is cyclic, 
  there exists an element $a \in A_{1}$ such that $a m_{i-1}=m_{i}$.
  Let $a'$ be an element satisfying $a'g =ga$.
  If $g m_{i-1}=0$, then $gm_{i}= g a m_{i-1} = a' g m_{i-1} =0$.
  Conversely, we assume that $g m_{i}=0$ and $g m_{i-1} \neq 0$.
  Let $b, b' \in A_{1}$ be elements satisfying $b g m_{i-1}=m_{i+n}$ and $bg =gb'$.
  The equality $m_{i+n}=b g m_{i-1}=g b' m_{i-1} =0$ leads to a contradiction.
\end{proof}
\dfnref{dfn-trunc-torfree} gives a useful property corresponding to a $g$-torsionfreeness of a module by \lemref{lem-trunc-basic}.
\begin{dfn}\label{dfn-trunc-torfree}
Let $A$ be a finitely presented connected $\bbN$-graded $k$-algebra generated in degree $1$ 
and $g$ a homogeneous normal element of $A$ with degree $n$.
For an integer $d \geq n$, a truncated point module $P=\bigoplus_{i=0}^{d}P_{i}$ of length $d+1$ over $A$ is called a \textit{truncated $g$-torsionfree point module of length $d+1$} over $A$ if the following equality is satisfied\textup{:}
\[
0=\left\{ a \in \Dbigoplus_{i=0}^{d-n}P_{i} \middle| ga=0 \right\}.
\]
\end{dfn}
\propref{prop-torfree-gP} is a generalization of well-known properties of a point module over $A$.
\begin{prop}\label{prop-torfree-gP}
Let $A$ be a finitely presented connected $\bbN$-graded $k$-algebra generated in degree $1$ 
and $P$ a truncated point module of length $d+1$ over $A$ \textup{(}resp. $P'$ a point module over $A$\textup{)}.
If $P$ is not a truncated $g$-torsionfree point module of length $d+1$ over $A$ \textup{(}resp. $P'$ is not a $g$-torsionfree point module over $A$\textup{)}, then $gP=0$ \textup{(}resp. $gP'=0$\textup{)}.
\end{prop}
\begin{proof}
  It follows from \lemref{lem-trunc-basic}.
\end{proof}
For a finitely presented connected $\bbN$-graded $k$-algebra $A$ generated in degree $1$, Artin--Tate--Van den Bergh \cite{ATV90} provided a useful relationship between the point scheme $\Gamma(A)$ of $A$ and isomorphism classes of point modules over $A$.
\begin{prop}[{\cite[Proposition 3.9, Corollary 3.13]{ATV90}}]\label{prop-ATV-corres}
  Let $A$ be a finitely presented connected $\bbN$-graded $k$-algebra generated in degree $1$ and $\Gamma_{d}(A)$ the $d$-th truncated point scheme of $A$ \textup{(}resp. $\Gamma(A)$ the point scheme of $A$\textup{)}.
  Then, a one-to-one correspondence exists between the set of isomorphism classes of truncated point modules of length $d+1$ over $A$ \textup{(}resp. point modules over $A$\textup{)} and 
  set of closed points of $\Gamma_{d}(A)$ \textup{(}resp. $\Gamma(A)$\textup{)}.
\end{prop} 
\subsection{Color Lie algebras}\label{subsec-color}
In this subsection, we recall the definitions and properties of the universal enveloping algebra of a color Lie algebra.
First, we recall the definitions of a skew symmetric bicharacter and a color Lie algebra.
\begin{dfn}[\cite{Ree}]
Let $G$ be an abelian group.
A map $\varepsilon : G \times G \rightarrow k^{\times}$ is a \textit{skew symmetric bicharacter}
if, for $\alpha, \beta, \gamma \in G$, the following equalities hold\textup{:}
\begin{enumerate}
  \item[(i)] $\varepsilon (\alpha, \beta)\varepsilon (\beta, \alpha)=1$;\\
\item[(ii)] $\varepsilon (\alpha+ \beta, \gamma)=\varepsilon   (\alpha, \gamma)\varepsilon (\beta, \gamma)$;\\
\item[(iii)] $\varepsilon (\alpha, \beta + \gamma)=\varepsilon (\alpha, \beta)\varepsilon (\alpha, \gamma)$.
\end{enumerate}
\end{dfn}
\begin{dfn}[\cite{Ree}]
  Let $G$ be an abelian group and $\varepsilon: G \times G \rightarrow k^{\times}$ be a skew symmetric bicharacter.
Suppose that $L=\bigoplus_{\gamma \in G}L_{\gamma}$ is a $G$-graded $k$-vector space and $[\,\, ,\,\, ] :L\times L\rightarrow L$ is a $G$-graded bilinear form.
A \textit{$( G, \varepsilon )$-color Lie algebra} is the pair $L=(L, [\,\, ,\,\, ])$ satisfying equalities for $a, b, c\in h(L)$ as follows\textup{:}
\begin{enumerate}
  \item[(i)] $[ a,b ] = -\varepsilon(|a|, |b|) [ b,a ]$;
  \item[(ii)] $\varepsilon(|c|, |a|)[ a, [ b, c ]] 
+\varepsilon(|a|, |b|)[ b, [ c, a ]]
+\varepsilon(|b|, |c|)[ c, [ a, b ]]=0$.
\end{enumerate}
Additionally, $k$-subspaces of $L$ are defined as $L_{-}:=\bigoplus_{\gamma \in G_{-}}L_{\gamma}$
(resp. $L_{+}:=\bigoplus_{\gamma \in G_{+}}L_{\gamma}$)
with 
$G_{+}:=\{ \gamma \in G \mid \varepsilon(\gamma, \gamma)=+ 1\}$
(resp. $G_{-}:=\{ \gamma \in G \mid \varepsilon(\gamma, \gamma)=- 1\}$). 
\end{dfn}
Next, for a $(G, \varepsilon)$-color Lie algebra $L$, we recall the definition of the universal enveloping algebra of $L$ and $\varepsilon$-symmetric algebra of $L$.
\begin{dfn}[\cite{Ree}]
Let $L$ be a $(G,\varepsilon)$-color Lie algebra, $T(L)$ the tensor algebra of $L$, and $J(L)$ the two-sided ideal of $T(L)$ generated by the relations $a \otimes b -\varepsilon(|a|, |b|)b\otimes a -[a, b]$ for any $a,b \in h(L)$.
The \textit{universal enveloping algebra} of $L$ is defined as follows\textup{:}
\[U(L):=T(L)/J(L).\]
\end{dfn}
The universal enveloping algebra $U(L)$ of a $(G, \varepsilon)$-color Lie algebra $L$ has natural $G$-gradation\textup{;} 
  the degree of $x_{1}\otimes \cdots \otimes x_{m}$
  is $|x_{1}|+\cdots +|x_{m}|$, 
  where
  $x_{i} \in h(L)$
   for 
  $1 \leq i \leq m$.

Also, for a $(G, \varepsilon)$-color Lie algebra $L$,
 the \textit{$\varepsilon$-symmetric algebra} $S_{\varepsilon}(L)$ of $L$ is the quotient of $T(L)$ by the two-sided ideal generated by the relations\textup{:}
\[
a \otimes b -\varepsilon(|a|, |b|)b\otimes a \text{ for } a,b \in h(L).
\]
For more details, see \cite{Sch}. 
Note that an $\varepsilon$-symmetric algebra can be defined for a $G$-graded $k$-vector space and a skew symmetric bicharacter $\varepsilon: G\times G \rightarrow k^{\times}$.
\begin{rem}\label{rem-grading}
  \begin{enumerate}
  \item 
  \blue{Let $U(L)$ be the universal enveloping algebra $U(L)$ of a finite dimensional color Lie algebra $L$.
  If $L_{-}=0$, then $U(L)$ has a finite Gelfand--Kirillov dimension by \cite[8.1.14]{MR} and \cite[4.C]{Sch} (also, see \cite{Ree})}.
  \item
  Let $U(L)$ be the universal enveloping algebra $U(L)$ of a finite dimensional $(\bbZ^{m}, \varepsilon)$-color Lie algebra $L$.
  $U(L)$ can be regarded as $\bbZ$-graded $k$-algebra as follows\textup{:}
 \[
  U(L) = \bigoplus_{i \in \bbZ}\left(\bigoplus_{i=a_{1}+ \cdots + a_{m}} U(L)_{(a_{1},\ldots,a_{m})}\right).
 \]
 If $L_{-}=0$ and the $\bbZ$-graded $k$-algebra $U(L)$ is a connected $\bbN$-graded $k$-algebra generated in degree $1$, 
 then $U(L)$ is an AS-regular algebra by \cite[6.3 Theorem]{Lev} and \cite[Theorems 3.1, 3.2]{Price97}.\label{rem-enu-grading}
 \end{enumerate}
 \end{rem}
At the end of this subsection, for the universal enveloping algebra $U(L)$ of a $(G, \varepsilon)$-color Lie algebra $L$, we recall a $G$-graded free resolution of the trivial module $k$ of $U(L)$ called the color Koszul resolution of $U(L)$ (\cite[Theorem 3]{CPvO}).
\begin{thm}\label{thm-color-Kos}
Let $L=(V, [\, ,\, ])$ be a $(G, \varepsilon)$-color Lie algebra and $U(L)$ the universal enveloping algebra of $L$.
  For $r \geq 0$, the free module $C_{r}$ is defined as $C_{r}=U(L)\otimes_{k}\bigwedge_{\varepsilon}^{r}V$, where $\bigwedge_{\varepsilon}^{r}V$ is the $r$-th part of $\bigwedge_{\varepsilon}=T(V)/\langle u\otimes v +\varepsilon(|u|,|v|)v\otimes u \mid u,v \in h(V) \rangle$.
  Also, $d_{r}:C_{r}\rightarrow C_{r-1}$ for $r \geq 1$ is defined as follows\textup{:}
  \begin{align*}
  a\otimes & (v_{1} \wedge \cdots \wedge v_{r})
  \\
  \mapsto&
  \dsum_{i=1}^{r}(-1)^{i+1}\eta_{i}a v_{i} \otimes \left( v_{1}\wedge \cdots \wedge \widehat{v_{i}}\wedge \cdots \wedge v_{r} \right)
  \\
  \quad&
  +\dsum_{1 \leq i < j \leq r}(-1)^{i+j}\eta_{i}\eta_{j}\varepsilon(|v_{j}|, |v_{i}|)a \otimes \left( [v_{i}, v_{j}]\wedge v_{1} \wedge \cdots\wedge \widehat{v_{i}} \wedge \cdots \wedge \widehat{v_{j}}\wedge \cdots \wedge v_{r}  \right)
  \end{align*}
  where $v_{i} \in h(L)$, $a \in h(U(L))$, $\eta_{1}=1$, and $\eta_{i}=\Pi_{l=1}^{i-1}\varepsilon(|v_{l}|,|v_{i}|)$ for $i \geq 2$.
  This symbol $\,\widehat{}\,$ indicates that the element below it must be omitted.
  Then, the sequence
  \[
     \cdots \rightarrow C_{i} \xrightarrow{d_{i}} C_{i-1} \rightarrow \cdots \rightarrow C_{2} \xrightarrow{d_{2}} C_{1} \xrightarrow{d_{1}} C_{0} \rightarrow k \rightarrow 0
  \]
  is a $G$-graded free resolution of the $G$-graded $U(L)$-module $k$.
\end{thm}
\section{$q'$-Heisenberg normal elements}\label{sec-gen}
In this section, for a $\bbZ$-graded $k$-algebra $A$, we define a \textit{$q'$-Heisenberg normal element} of $A$ and prove some of its basic properties.  
\begin{dfn}\label{dfn-one}
Let $A$ be a $\bbZ$-graded $k$-algebra.
A regular homogeneous normal element $g$ of $A$ with degree $n \geq 1$ is called a \textit{$q'$-Heisenberg normal element} of $A$ if there exist $u \in k^{\times}$, $x \in A_{1}$, and $y \in A_{n-1}$ that satisfy the following equalities\textup{:}
\begin{enumerate}
  \item[(i)] $g=xy-uyx$;
  \item[(ii)] $0=xg-ugx$;
  \item[(iii)] $0=gy-uyg$.
\end{enumerate}
\end{dfn}
\begin{rem}
  If either $x$ or $y$ is a central element of $A$, then $g=0$ holds.
  In particular, when $A$ is commutative, $A$ has no $q'$-Heisenberg normal element.
\end{rem}
For $\bbZ$-graded $k$-algebra $A$, in certain situations, a $q'$-Heisenberg normal element of $A$ is not necessarily regular, but $u \in k^{\times}$ is sufficient.
For convenience, we assume the regularity of a $q'$-Heisenberg normal element.
We show some examples of a $q'$-Heisenberg normal element of some algebras.
\begin{ex}
  \begin{enumerate}
    \item
    For $\alpha, \beta \in k$, suppose that $A(\alpha, \beta)$ is a graded down-up algebra
    \[
    A(\alpha, \beta):=k\langle x,y \rangle/\langle x^{2}y-\alpha xyx-\beta yx^{2}, xy^{2}-\alpha yxy-\beta y^{2}x\rangle
    \]
    (see \cite{IU}).
    If $\alpha, \beta \neq 0$ and $\alpha^{2}+4\beta=0$ hold, then $xy-\dfrac{\alpha}{2}yx$ is a $q'$-Heisenberg normal element of $A(\alpha, \beta)$.
    This algebra is isomorphic to another $q$-analog of the universal enveloping algebra of the Heisenberg algebra $H'_{q}$ with $q=\dfrac{\alpha}{2}$ as in \cite[Proposition 2.4]{KS}.
    \item 
    Another example of a $q'$-Heisenberg normal element can be found in \cite[Theorem A (d)]{LPWZ}.
    For $v,p \in k$ and $p \neq 0$, let $D(v,p)$ be the algebra defined as follows\textup{:}
    \[
    D(v,p) := k\langle x,y \rangle\left/ \middle\langle \begin{aligned}
      &xy^{2} + vyxy + p^{2}y^{2}x,\\
      &x^{3}y + (v + p)x^{2}yx + (p^{2} +pv)xyx^{2} +p^{3}yx^{3}
    \end{aligned}
    \right\rangle.
    \]
    If $v=2p$ holds, then an element $x^{2}y+2pxyx+p^{2}yx^{2}$ is a $q'$-Heisenberg normal element of $D(v, p)$.
    \item
    More generally, for an $\bbN$-graded $k$-algebra $A$, let $f$ be a homogeneous element with degree $n$ and $\delta$ a graded derivation of $A$.
    Suppose that $B=A[X, \delta]$ is a graded Ore extension of $A$ with a degree $1$ element $X$.
    If $Y=\delta(f)$ is regular in $B[Y]/\langle Y-\delta(f) \rangle$, where $Y$ is a degree $n+1$ element, then $\delta(f)$ is a $q'$-Heisenberg normal element of it.
    Indeed, the definition of the Ore extension gives the following equalities\textup{:}
    \begin{align*}
     \begin{cases}
      Xf=fX+\delta(f) &\text{ in } A[X, \delta];\\
      Yr=rY &\text{ for } r \in B[Y] \text{ in } B[Y]=A[X, \delta][Y].
     \end{cases}
    \end{align*}
  \end{enumerate}
\end{ex}
We will show some basic properties of a $q'$-Heisenberg normal element.
First, we recall a basic property of a normal element.
\begin{lem}\label{lem-normal}
  Let $A$ be a $\bbZ$-graded $k$-algebra, $g$ a regular homogeneous normal element of $A$ with degree $n \geq 1$, and $A^{[n]}$ the $n$-th quasi-Veronese algebra of $A$. 
  Suppose that $\bfg$ is the element of $A^{[n]}$ defined as the following equality\textup{:}
  \[
  \bfg :=
  (\delta_{i,j} g)
  =
  \left(
    \begin{array}{cccc}
      g& & & \\
       &g& & \\
       & &\ddots&\\
       &&&g\\
    \end{array}
  \right)
  \in
  \left(
    \begin{array}{cccc}
      A_{n}&A_{n+1}&\cdots&A_{2n-1}\\
      A_{n-1}&A_{n}& &\vdots\\
      \vdots& &\ddots&\vdots\\
      A_{1}&\cdots&\cdots&A_{n}\\
    \end{array}
  \right)
  =(A^{[n]})_{1},
  \]
  where $\delta_{i,j}$ is the Kronecker delta. 
  Then, $\bfg$ is a regular homogeneous normal element of $A^{[n]}$ with degree $1$.
\end{lem}
\begin{proof}
  Let $\nu_{g}$ be the graded automorphism determined by $\nu_{g}(a)g=ga$ for $a \in A$.
  Also, the map defined by $\nu_{\bfg}^{[n]}((a_{i,j}))=(\nu_{g}(a_{i,j}))$ for $(a_{i,j}) \in (A^{[n]})_{r}$ is a graded automorphism of $A^{[n]}$ in \cite{Mori}.
  For $\mathbf{a}=(a_{i,j}) \in (A^{[n]})_{r}$, the following equality holds\textup{:}
  \[
  \mathbf{ga}
  =\left( \sum _{l=0}^{n-1}\delta_{l,j}g a_{i,l} \right)
  =(ga_{i,j})
  =(\nu_{g}(a_{i,j})g)
  =\left( \sum _{l=0}^{n-1}\nu_{g}(a_{l,j})\delta_{i,l}g \right)
  =\nu_{\bfg}^{[n]}(\mathbf{a})\mathbf{g}.
  \] 
  Then, $\bfg$ is a regular normal element in $A^{[n]}$ with degree $1$.
\end{proof}
 Note that, by \cite{ATV91} and \cite[Example 2.12]{Zhang}, the set $\nu_{\bfg}^{[n]}=\{(\nu_{\bfg}^{[n]})^{i}\mid i \in \bbZ \}$ is a twisting system of $A^{[n]}$, and $\bfg$ is a central element of the twisted algebra ${}^{\nu_{\bfg}^{[n]}}A^{[n]}$.
\begin{rem}
  In the remainder of this section, we use the following notations.
  For a $\bbZ$-graded $k$-algebra $A$, let $g$ be a regular homogeneous normal element of $A$ with degree $n \geq 1$.
 We denote an element $\bfg:=(\delta_{i,j}g)$ of $A^{[n]}$ as $\bfg$, where $\delta_{i,j}$ is the Kronecker delta.
 Also, let $\nu_{g}$ denote the graded automorphism of $A$ defined by $\nu_{g}(a)g=ga$ for $a \in A$.
Let $\nu_{\bfg}^{[n]}$ denote the graded automorphism of $A^{[n]}$ defined by $\nu_{\bfg}^{[n]}((a_{i,j}))=(\nu_{g}(a_{i,j}))$ for $(a_{i,j}) \in (A^{[n]})_{r}$.
\end{rem}

We consider a relationship between an algebra having a $q'$-Heisenberg normal element and the first Weyl algebra. 
The \textit{first Weyl algebra} $W_{1}(k)$ is defined by
\[
W_{1}(k):=k\langle X, Y \rangle/\langle XY-YX-1 \rangle
\]
(see \cite{Co}).
Also, a filtration $\mathcal{B}=\{\mathcal{B}_{i}\}$ on $W_{1}(k)$ is called the \textit{Bernstein filtration} defined by
  \[
  \mathcal{B}_{i}:=
  \begin{cases}
  \left\{\dsum_{|\alpha|+|\beta|\leq i} c_{\alpha,\beta}Y^{\alpha}X^{\beta} \in W_{1}(k) \middle| c_{\alpha,\beta} \in k \right\} &\text{for} \, i\geq 0;\\
  0 &\text{for}\, i<0;
  \end{cases}
  \]  
  (see \cite{Co}).
\begin{prop}\label{prop-Weyl-hom}
  Let $A$ be a $\bbZ$-graded $k$-algebra, $g=xy-uyx$ be a $q'$-Heisenberg normal element of $A$ with degree $n$, and ${}^{\nu_{\bfg}^{[n]}}(A^{[n]})$ be the twisted algebra of the $n$-th quasi-Veronese algebra of $A$ by $\nu_{\bfg}^{[n]}$.
  Assume that $\mathcal{R}$ is the dehomogenization of ${}^{\nu_{\bfg}^{[n]}}(A^{[n]})$ with respect to $\bfg$ and $F\mathcal{R}=\{F_{i}\mathcal{R}\}_{i \in \bbZ}$ is the filtration of $\mathcal{R}$ defined by the dehomogenization.
  Then, there is a homomorphism $W_{1}(k)\rightarrow \mathcal{R}$ such that it preserves the filtrations $\mathcal{B}$ and $F\mathcal{R}$.
\end{prop}
\begin{proof}
  Let $\varphi : k\langle X,Y \rangle\rightarrow \mathcal{R}$ be the $k$-algebra homomorphism defined by
  \begin{align*}
  &X\mapsto
  \left(
  \begin{array}{cccc}
    &xg&&\\
    &&\ddots&\\
    &&&xg\\
    x&&&\\
  \end{array}
  \right)
  +\langle \bfone-\bfg \rangle
  \in
  (A^{[n]})_{1}
  +\langle \bfone-\bfg \rangle
  =F_{1}\mathcal{R}
  ,\\
  &Y \mapsto
  \left(
  \begin{array}{cccc}
    &&&gy\\
    y&&&\\
    &\ddots&&\\
    &&y&\\
  \end{array}
  \right)
  +\langle \bfone-\bfg \rangle
  \in
  (A^{[n]})_{1}
  +\langle \bfone-\bfg \rangle
  =F_{1}\mathcal{R}
  ,
  \end{align*}
  where $\bfone$ is the identity of $A^{[n]}$.
We will show that $\varphi(XY-YX-1)=0$.
Since $\nu_{g}(x)=u^{-1}x$, $\nu_{g}(y)=uy$, and $\nu_{g}(g)=g$, the following equalities are obtained\textup{:}
\allowdisplaybreaks[4]
\begin{align*}
  \varphi(X)\varphi(Y)&=
  \left(
  \begin{array}{cccc}
    &xg&&\\
    &&\ddots&\\
    &&&xg\\
    x&&&\\
  \end{array}
  \right)
  \circ
  \left(
  \begin{array}{cccc}
    &&&gy\\
    y&&&\\
    &\ddots&&\\
    &&y&\\
  \end{array}
  \right)
  +\langle \bfone-\bfg \rangle
  \\
  &=
  u^{-1}
  \left(
  \begin{array}{cccc}
    &xg&&\\
    &&\ddots&\\
    &&&xg\\
    x&&&\\
  \end{array}
  \right)
  \left(
  \begin{array}{cccc}
    &&&gy\\
    y&&&\\
    &\ddots&&\\
    &&y&\\
  \end{array}
  \right)
  +\langle \bfone-\bfg \rangle
  \\
  &=
  u^{-1}
  \left(
  \begin{array}{cccc}
    xgy&&&\\
    &xgy&&\\
    &&\ddots&\\
    &&&xgy\\
  \end{array}
  \right)
  +\langle \bfone-\bfg \rangle
  \\
  &=
  \bfg
  \left(
  \begin{array}{cccc}
    xy&&&\\
    &xy&&\\
    &&\ddots&\\
    &&&xy\\
  \end{array}
  \right)
  +\langle \bfone-\bfg \rangle
\end{align*}
and
\begin{align*}
\varphi(Y)\varphi(X)
  &=
  \left(
  \begin{array}{cccc}
    &&&gy\\
    y&&&\\
    &\ddots&&\\
    &&y&\\
  \end{array}
  \right)
  \circ
  \left(
  \begin{array}{cccc}
    &xg&&\\
    &&\ddots&\\
    &&&xg\\
    x&&&\\
  \end{array}
  \right)
  +\langle \bfone-\bfg \rangle
  \\
  &=
  u
  \left(
  \begin{array}{cccc}
    &&&gy\\
    y&&&\\
    &\ddots&&\\
    &&y&\\
  \end{array}
  \right)
  \left(
  \begin{array}{cccc}
    &xg&&\\
    &&\ddots&\\
    &&&xg\\
    x&&&\\
  \end{array}
  \right)
  +\langle \bfone-\bfg \rangle
  \\
  &=
  u
  \left(
  \begin{array}{cccc}
    yxg&&&\\
    &\ddots&&\\
    &&yxg&\\
    &&&gyx\\
  \end{array}
  \right)
  +\langle \bfone-\bfg \rangle
  \\
  &=
  u\bfg
  \left(
  \begin{array}{cccc}
    yx&&&\\
    &\ddots&&\\
    &&yx&\\
    &&&yx\\
  \end{array}
  \right)
  +\langle \bfone-\bfg \rangle.
\end{align*}
Consequently, the following equality is obtained\textup{:}
\begin{align*}
&\varphi(XY-YX-1)=\varphi(X)\varphi(Y)-\varphi(Y)\varphi(X)-\bfone\\
&\quad=
  \bfg
  \left(
  \begin{array}{cccc}
    xy&&&\\
    &\ddots&&\\
    &&xy&\\
    &&&xy\\
  \end{array}
  \right)
  -u\bfg
  \left(
  \begin{array}{cccc}
    yx&&&\\
    &yx&&\\
    &&\ddots&\\
    &&&yx\\
  \end{array}
  \right)
  -\bfone
  +\langle \bfone-\bfg \rangle
  =0.
\end{align*}
\allowdisplaybreaks[0]
Therefore, $\varphi$ factors through $W_{1}(k)$.
Obviously, the induced homomorphism preserves the filtrations $\mathcal{B}$ and $F\mathcal{R}$.
\end{proof}
Since the first Weyl algebra $W_{1}(k)$ is simple (see \cite[2.2.1 Theorem]{Co}), $\mathcal{R}$ has a subalgebra which is isomorphic to $W_{1}(k)$ by \propref{prop-Weyl-hom}.
\begin{prop}\label{prop-fat-point}
  Let $A$ be a finitely presented connected $\bbN$-graded $k$-algebra generated in degree $1$, and $g=xy-uyx$ be a $q'$-Heisenberg normal element of $A$ with degree $n$. 
  Then, a one-to-one correspondence exists between the set of isomorphism classes of base-point modules over $A$ and that of $A/\langle g \rangle$.
\end{prop}
\begin{proof}
  Let $M$ be a base-point module satisfying $\dimk_{k}M_{i}=c$ for $i \geq 0$ and $\dimk_{k}M_{i}=0$ for $i <0$.
  Assume that $M$ is $g$-torsionfree.
  Let $A^{[n]}$ be the $n$-th quasi-Veronese algebra of $A$ and ${}^{\nu_{\bfg}^{[n]}}(A^{[n]})$ the twisted algebra of $A^{[n]}$ by $\nu_{\bfg}^{[n]}$.
  By \lemref{lem-quasi-Vero} and \propref{prop-Zhang-twi}, the following categorical equivalence holds\textup{:}
\[ 
A\lGrMod \cong A^{[n]}\lGrMod \cong {}^{\nu_{\bfg}^{[n]}}(A^{[n]})\lGrMod.
\]
Additionally, ${}^{\nu_{\bfg}^{[n]}}(Q(M))$ is a $\bfg$-torsionfree graded ${}^{\nu_{\bfg}^{[n]}}(A^{[n]})$-module. 
Let $\mathcal{N}$ be the dehomogenization of ${}^{\nu_{\bfg}^{[n]}}(Q(M))$ with respect to $\bfg$ and $F\mathcal{N}=\{F_{i}\mathcal{N}\}_{i \in \bbZ}$ the filtration of $N$ defined by the dehomogenization.
Then, by \propref{prop-Weyl-hom}, ${}^{\nu_{\bfg}^{[n]}}(Q(M))$ is also a filtered $W_{1}(k)$-module.
Furthermore, by \cite[10.3.1]{Co} and the following equality, $\mathcal{N}$ is a holonomic module of $W_{1}(k)$ (see \cite{Co})\textup{:}
\[
\dimk_{k}F_{s}\mathcal{N} 
= \dimk_{k} {}^{\nu_{\bfg}^{[n]}}(Q(M))_{s}
= \dsum_{i=0}^{n-1}\dimk_{k}M_{sn-i}
= cn
\quad\text{  for $s > 0$.}
\]
Hence, a good filtration $\Gamma=\{\Gamma_{i}\}_{i \in \bbZ}$ of $\mathcal{N}$ and $c_{1} \neq 0, c_{0} \in \mathbb{Q}$ exists such that $\dimk_{k}\Gamma_{s}=c_{1}s+c_{0}$ for $s \gg 0$. 
Furthermore, $s_{0} \in \bbZ$ exists such that $\Gamma_{i} \subset F_{i+s_{0}}\mathcal{N}$ for $i \geq 0$ by \cite[8.3.2]{Co}.
Hence, the inequality
$cn = \dimk_{k}F_{s+s_{0}}\mathcal{N} \geq \dimk_{k}\Gamma_{s} = c_{1}s+c_{0}$ holds for $s \gg 0$.
This is a contradiction. 
Therefore, $M$ is not $g$-torsionfree, and $gM=0$ by \cite[Lemma 5]{cCV}.
\end{proof}
\begin{rem}
  Let $M=\bigoplus_{i \in \bbZ} M_{i}$ be a $g$-torsionfree left bounded graded module such that $\dimk_{k} M_{s}=c_{1}s+c_{0}$ for $s \gg 0$, where $c_{1}\neq 0$, $c_{0} \in \mathbb{Q}$ (e.g., $M$ is a line module).
  By the similar proof of \propref{prop-fat-point}, the dehomogenization of ${}^{\nu_{\bfg}^{[n]}}Q(M)$ is a holonomic module of $W_{1}(k)$.
\end{rem}
\section{$q'$-Heisenberg normal elements and point modules}\label{sec-trunc}
 In this section, for a finitely presented connected $\bbN$-graded $k$-algebra $A$ generated in degree $1$, we show more details of a relationship between a point scheme of $A$ and a $q'$-Heisenberg normal element of $A$.
At the end of this section, we show the structure of the set of closed points of the point scheme of the universal enveloping algebra of a color Lie algebra.
\subsection{$q'$-Heisenberg normal elements and truncated point modules}\label{subsec-trunc}
Throughout this subsection, let $A$ be a finitely presented connected $\bbN$-graded $k$-algebra generated in degree $1$ unless otherwise specified.
In this subsection, we show a relationship between a $q'$-Heisenberg normal element $g$ of $A$ and a truncated $g$-torsionfree point module over $A$.

First, we recall useful equalities obtained by the definition of a $q'$-Heisenberg normal element of $A$.
 For a $q'$-Heisenberg normal element $g=xy-uyx$ of $A$, the equalities $x^{2}y-2uxyx+u^{2}yx^{2}=0$ and $xy^{2}-2uyxy+u^{2}y^{2}x=0$ hold.
 Then, the same equalities in a down-up algebra hold as in \cite{IU}.
\begin{lem}\label{lem_IU-lem2.3}
 Let $A$ be a $\bbZ$-graded $k$-algebra \textup{(}not necessarily connected, $\bbN$-graded, finitely presented, or generated in degree $1$\textup{)}, and $g=xy-uyx$ be a $q'$-Heisenberg normal element of $A$ with degree $n$.
 Then, for an integer $r \geq 1$, the following equalities hold\textup{:}
 \begin{enumerate}
  \item[\textup{(1)}] $x^{r}y=ru^{r-1}xyx^{r-1}-(r-1)u^{r}yx^{r}$\textup{;}
  \item[\textup{(2)}] $yx^{r}=ru^{-(r-1)}x^{r-1}yx-(r-1)u^{-r}x^{r}y$.
 \end{enumerate}
\end{lem}
\begin{proof}
See \cite[Lemma 2.3]{IU}.
\end{proof}
The display of a $q'$-Heisenberg normal element $g$ of $A$ is not necessarily unique, that is, we may have displays $g=xy-uyx=x'y'-u'y'x'$, where $x, x' \in A_{1}$, $y,y' \in A_{n-1}$, and $u, u' \in k^{\times}$.
However, when $g$ is displayed by $g=xy-uyx$, a $g$-torsionfree truncated point module over $A$ has a useful relationship with $x$.
\begin{prop}\label{prop-tru-torfree}
  Let $A$ be a finitely presented connected $\bbN$-graded $k$-algebra generated in degree $1$, $g=xy-uyx$ be a $q'$-Heisenberg normal element of $A$ with degree $n$ of $A$, and $P=\bigoplus_{i = 0}^{d}k m_{i}$ be a truncated $g$-torsionfree point module of length $d + 1$ over $A$, where $d \geq 2n-1$.
  Then, $x m_{i} \neq 0$ holds for $0 \leq i \leq d-1$.
\end{prop}
\begin{proof}
  Let $\{p_{i+1}\}_{i=0}^{d-1}$ (resp. $\{q_{i+(n-1)}\}_{i=0}^{d-(n-1)}$, $\{\lambda_{i+n}\}_{i=0}^{d-n}$) be a sequence defined by $p_{i+1}m_{i+1}=x m_{i}$ (resp. $q_{i+(n-1)}m_{i+(n-1)}=y m_{i}$, $\lambda_{i+n}m_{i+n}=g m_{i}$).
  For $1 \leq i \leq d-n$, the following equality holds\textup{:}
  \[
  p_{i+n}\lambda_{i+n-1}m_{i+n}=xgm_{i-1}=ugxm_{i-1}=u\lambda_{i+n}p_{i}m_{i+n}.
  \]
  Since $P$ is a truncated $g$-torsionfree point module over $A$, $\lambda_{i+n-1} \neq 0$ and $\lambda_{i+n}\neq 0$.
  Then, the equality $p_{i+n}=0$ holds if and only if $p_{i}=0$ holds for $1 \leq i \leq d-n$. 

  The set $\{ i \mid p_{i}=0 \text{ for } 1 \leq i \leq n \}$ is denoted as $B$.
  We will show that $B$ is empty.
  First, assume that the number of elements in $B$ is two or more.
  Let $(i,j)$ be a pair of elements in $B$ such that $j > i$ and minimize $j-i$.
  \begin{enumerate}
  \item The case that $j-i=1$ holds.
  A contradiction is obtained from the following equality\textup{:}
  \[
  0 \neq gm_{i} = (xy-uyx)m_{i} = (p_{i+n}q_{i+n-1}-uq_{i+n}p_{i+1})m_{i+n}=0.
  \]
  \item The case that $j-i\geq 2$ holds.
  \lemref{lem_IU-lem2.3} gives the following\textup{:}
  \begin{align*}
  0
  &=q_{j+n-1}p_{j}p_{j-1} \cdots p_{i+1}m_{j+n-1}\\
  &=yx^{j-i}m_{i}\\
  &=((j-i)u^{-(j-i-1)}x^{j-i-1}yx-(j-i-1)u^{-(j-i)}x^{j-i}y)m_{i}\\
  &=((j-i)u^{-(j-i-1)}p_{j+n-1} \cdots p_{i+n+1} q_{i+n} p_{i+1}\\
  &\quad
  -(j-i-1)u^{-(j-i)}p_{j+n-1} \cdots p_{i+n+1} p_{i+n} q_{i+n-1})m_{j+n-1}\\
  &=(j-i)u^{-(j-i-1)}p_{j+n-1} \cdots p_{i+n+1} q_{i+n} p_{i+1}m_{j+n-1}.
  \end{align*}
  It follows that $q_{i+n}=0$ by the choice of $(i, j)$, and the following equality yields a contradiction\textup{:}
  \[
  0 \neq gm_{i} = (xy-uyx)m_{i} = (p_{i+n}q_{i+n-1}-uq_{i+n}p_{i+1})m_{i+n}=0.
  \]
  \end{enumerate}
  Therefore, the number of elements in $B$ is less than two.
  Assume that $B = \{ i \}$, where $1 \leq i \leq n$.
  If $i = n$, then the following equality holds\textup{:}
  \begin{align*} 
  0
  &=q_{2n-1} p_{n} \cdots p_{1} m_{2n-1}\\
  &=yx^{n} m_{0}\\
  &=(nu^{-(n-1)}x^{n-1}yx-(n-1)u^{-n}x^{n}y)m_{0}\\
  &=(nu^{-(n-1)}p_{2n-1} \cdots p_{n+1} q_{n} p_{1}
    -(n-1)u^{-n}p_{2n-1} \cdots p_{n+1} p_{n} q_{n-1})m_{2n-1}\\
  &=nu^{-(n-1)}p_{2n-1} \cdots p_{n+1} q_{n} p_{1} m_{2n-1}.\\
  \end{align*}
  It follows that $q_{n}=0$, and a contradiction is obtained from the following equality\textup{:}
  \[
  0 \neq gm_{0} 
  = (xy-uyx)m_{0} 
  = (p_{n} q_{n-1} -u q_{n}p_{1})m_{n}
  =0.
  \]
  Hence, $i \neq n$, that is, $p_{n} \neq 0$.
  When $i \neq n-1$, \lemref{lem_IU-lem2.3} gives the following\textup{:}
  \begin{equation}\label{equation_0}
  \begin{aligned}
  0
  &=p_{2n-1} \cdots p_{i+n} q_{i+n-1} m_{2n-1}\\
  &=x^{n-i}y m_{i}\\
  &=((n-i)u^{n-i-1}xyx^{n-i-1}-(n-i-1)u^{n-i}yx^{n-i})m_{i}\\
  &=((n-i)u^{n-i-1}p_{2n-1} q_{2n-2} p_{n-1} \cdots p_{i+1}\\
  &\quad
   -(n-i-1)u^{n-i}q_{2n-1} p_{n} p_{n-1} \cdots p_{i+1})m_{2n-1}.\\
  \end{aligned}
  \end{equation}
  It follows that
  \begin{equation*}
    (n-i)p_{2n-1} q_{2n-2} q_{n-1} -(n-i-1)u q_{2n-1} p_{n} q_{n-1} =0.
  \end{equation*}
  When $i=n-1$, the same equality also holds.
  Similarly, the following equality is obtained\textup{:}
   \begin{equation*}
   iu^{2}q_{2n-1}q_{n}p_{1}-(i-1)uq_{2n-1}p_{n}q_{n-1}=0.
   \end{equation*}
  Also, the following equality is obtained\textup{:}
  \begin{align*}
    0&=(xy^{2}-2uyxy+u^{2}y^{2}x)m_{0}\\
    &=(p_{2n-1}q_{2n-2}q_{n-1}-2uq_{2n-1}p_{n}q_{n-1}+u^{2}q_{2n-1}q_{n}p_{1})m_{2n-1}.
  \end{align*}
  Consequently, the above equality yields
  \[
  \left(
  \begin{array}{ccc}
    (n-i) & -(n-i-1) & 0 \\
    0 & -(i-1) & i \\
    1 & -2 & 1 \\
  \end{array}
  \right) 
  \left(
  \begin{array}{c}
    p_{2n-1}q_{2n-2}q_{n-1}\\
    uq_{2n-1}p_{n}q_{n-1}\\
    u^{2}q_{2n-1}q_{n}p_{1}
  \end{array}
  \right)
  =0.
  \]
  Since 
  $\mathrm{det}
  \left(
  \begin{array}{ccc}
    (n-i) & -(n-i-1) & 0 \\
    0 & -(i-1) & i \\
    1 & -2 & 1 \\
  \end{array}
  \right) 
  = n$, 
  $p_{2n-1}q_{2n-2}q_{n-1}$, $q_{2n-1}p_{n}q_{n-1}$, $q_{2n-1}q_{n}p_{1}$
  are zero.
  If $q_{2n-1} \neq 0$, then $p_{n}q_{n-1}=q_{n}p_{1}=0$, and a contradiction is obtained from the following equality\textup{:} 
  \[
  0 \neq gm_{0}
  = (xy-uyx)m_{0}
  = (p_{n}q_{n-1}-uq_{n}p_{1})m_{n}
  =0.
  \]
  Therefore, $q_{2n-1} = 0$.
  When $i \neq n-1$, equality \eqref{equation_0} gives the following\textup{:}
  \begin{align*}
   0
   &=(n-i)u^{n-i-1}p_{2n-1} q_{2n-2} p_{n-1} \cdots p_{i+1}
   -(n-i-1)u^{n-i}q_{2n-1} p_{n} p_{n-1} \cdots p_{i+1}\\
   &=(n-i)u^{n-i-1}p_{2n-1} q_{2n-2} p_{n-1} \cdots p_{i+1}.
  \end{align*}
  It follows that $p_{2n-1}q_{2n-2}=0$. 
  When $i=n-1$, the same equality also holds. 
  Then, 
  \[
  0 \neq gm_{n-1} =(xy-uyx)m_{n-1}
  =(p_{2n-1}q_{2n-2}-uq_{2n-1}p_{n})m_{2n-1}=0.
  \]
  This is a contradiction.
  Hence, $B$ is empty, that is, $xm_{i} \neq 0$ for $0 \leq i \leq d-1$.
\end{proof}
\propref{prop-tru-torfree} shows the nonexistence of a long truncated $g$-torsionfree point module over $A$, where $g$ is a $q'$-Heisenberg normal element of $A$.
\begin{thm}\label{thm-tru-pmod}
  Let $A$ be a finitely presented connected $\bbN$-graded $k$-algebra generated in degree $1$, and $g=xy-uyx$ be a $q'$-Heisenberg normal element of $A$ with degree $n$.
  For $d \geq 2n-1$, 
  $A$ has no truncated $g$-torsionfree point module of length $d+1$ over $A$.
\end{thm}
\begin{proof}
The notation from the proof of \propref{prop-tru-torfree} is retained.
Let $\{r_{i}\}_{i=1}^{n+1}$ be a sequence defined by 
$r_{i}=\dfrac{q_{i+n-2}}{p_{i+n-2}p_{i+n-3}\cdots p_{i}}$
and $P=\bigoplus_{i = 0}^{d}k m_{i}$, a truncated $g$-torsionfree point module of length $d + 1$ over $A$, where $d \geq 2n-1$.
The definition of a $q'$-Heisenberg normal element of $A$ gives the following equalities\textup{:}
\[
0=(x^{n-1}g-u^{n-1}gx^{n-1})m_{0}=(p_{2n-1} \cdots p_{n+1}\lambda_{n}-u^{n-1}\lambda_{2n-1}p_{n-1} \cdots p_{1})m_{2n-1}
\]
and
\[
0=(yg-u^{-1}gy)m_{0}=(q_{2n-1}\lambda_{n}-u^{-1}\lambda_{2n-1}q_{n-1})m_{2n-1}.
\]
Then, the equality $r_{n+1}=u^{-n}r_{1}$ holds.
%
Meanwhile, for $1 \leq i \leq n-1$, 
\begin{align*}
   0&=p^{-1}_{i+n}p^{-1}_{i+n-1}\cdots p^{-1}_{i}
      (x^{2}y-2uxyx+u^{2}yx^{2})m_{i-1}\\
    &=p^{-1}_{i+n}p^{-1}_{i+n-1}\cdots p^{-1}_{i}
       (p_{i+n}p_{i+n-1}q_{i+n-2}-2up_{i+n}q_{i+n-1}p_{i} +u^{2}q_{i+n}p_{i+1}p_{i})m_{i+n}\\
    &=(r_{i}-2ur_{i+1} +u^{2}r_{i+2})m_{i+n}.
\end{align*}
It follows that $r_{i+1}-u^{-1}r_{i}=u^{1-i}(r_{2}-u^{-1}r_{1})$ for $1 \leq i \leq n$ and then
\begin{align*}
 0
 &=r_{n+1}-u^{-n}r_{1}
 =\dsum^{n}_{i=1}u^{i-n}(r_{i+1}-u^{-1}r_{i})\\
 &=\dsum^{n}_{i=1}u^{i-n}u^{1-i}(r_{2}-u^{-1}r_{1})
 =nu^{1-n}(r_{2}-u^{-1}r_{1}).
\end{align*}
Consequently, the following equality leads to a contradiction\textup{:}
\[
0 \neq gm_{0}
  =(p_{n}q_{n-1}-uq_{n}p_{1})m_{n}
  =p_{1}p_{2}\cdots p_{n} (r_{1}-ur_{2})m_{n}
  =0.
\]
\end{proof}
We remark that, since a truncated $g$-torsionfree point module over $A$ is not a $g$-torsionfree module over $A$, \thmref{thm-tru-pmod} is not a corollary of \propref{prop-fat-point}.
\subsection{Point modules over the universal enveloping algebras of color Lie algebras}\label{sec-color}
Throughout this subsection, for integers $m \geq 1$ and $m \geq i \geq 0$, let  $e_{i} \in \bbZ^{m+1}$ denote $(\delta_{i,0}, \delta_{i, 1} , \ldots , \delta_{i, m})$, where $\delta_{i,j}$ is the Kronecker delta.
Also, for a finite dimensional $(\bbZ^{m+1}, \varepsilon)$-color Lie algebra $L$ with $L_{-}=0$, 
let $L_{\bfone}$ denote $\bigoplus_{0 \leq i \leq m}L_{e_{i}}$ and $n_{L}$ denote $n_{L}=\mathrm{max}\{j \mid L_{\bfone}^{j} \neq 0 \}$, where $L_{\bfone}^{1}=L_{\bfone}$ and $L_{\bfone}^{j+1}=[L_{\bfone}^{j},L_{\bfone}]$ for $j \geq 1$.
Furthermore, the map $|\quad|_{\bbZ}:h(L)\setminus \{0\} \rightarrow \bbZ$ is defined as $L_{(a_{0},\ldots, a_{m})} \ni x \mapsto a_{0}+ \cdots + a_{m} \in \bbZ$.
When $U(L)$ is considered as a $\bbZ$-graded $k$-algebra based on \remref{rem-grading} \eqref{rem-enu-grading}, let $\deg x \in \bbZ$ denote the degree of $x \in h(U(L))$.

In this subsection, we use \thmref{thm-tru-pmod} on the universal enveloping algebra of a color Lie algebra.
First, we recall the definition of a skew polynomial algebra on $m+1$ variables
(see \cite{BDL,HvO,MR,Vit} for details)\textup{:}
  \[
  S_{\omega}=k\langle x_{0}, x_{1}, \ldots ,x_{m} \rangle/\langle x_{i}x_{j}-\omega_{i,j}x_{j}x_{i}\mid 0 \leq i,j \leq m\rangle
  \] 
  where $\omega=(\omega_{i,j})$ is an $(m+1)\times (m+1)$  matrix over $k$ such that $\omega_{i,j}\omega_{j,i}=1$ and  $\omega_{l, l}=1$ for $0 \leq i,j,l \leq m$.
In the following, by assuming that the degrees of $x_{i}$ are $1$, we consider $S_{\omega}$ as a $\bbZ$-graded $k$-algebra generated in degree $1$.
We recall an important result of the set of isomorphism classes of point modules over a skew polynomial algebra.
\begin{thm}[{\cite[Theorem 1]{BDL}}, \cite{Vit}]\label{thm-BDL}
  Let $S_{\omega}$ be a skew polynomial algebra on $m+1$ variables with $\omega = (\omega_{i,j})$ for $m \geq 2$.
  Then, a one-to-one correspondence exists between the set of isomorphism classes of point modules over $S_{\omega}$ and 
  \[
  \{(p_{0}:p_{1}:\cdots:p_{m})\in \bbP^{m}\mid (\omega_{i,j}\omega_{j,l}-\omega_{i,l})p_{i}p_{j}p_{l}=0 \text{ for } 0\leq i < j < l \leq m \}.
  \]
\end{thm}
Also, when $m=1$, a one-to-one correspondence exists between the set of isomorphism classes of point modules over $S_{\omega}$ and set of closed points of $\bbP^{1}$ by \cite{ATV90}.
By those facts and \thmref{thm-tru-pmod}, we obtain the structure of the set of isomorphism classes of point modules over the universal enveloping algebra of color Lie algebras.

We remark that, for a finite dimensional $(\bbZ^{m+1}, \varepsilon)$-color Lie algebra $L$ with $L_{-}=0$ generated by $\{\theta_{0},\theta_{1}, \ldots , \theta_{m} \}$ with $|\theta_{i}|=e_{i}$ for $0 \leq i \leq m$, 
$S_{\varepsilon}(L_{\bfone})$ is a skew polynomial algebra 
since $S_{\varepsilon}(L_{\bfone})\cong S_{\varepsilon}(L/\bigoplus_{|a|_{\bbZ}\geq 2} L_{a})$.
\begin{thm}\label{thm-color-var}
  Let $L$ be a finite dimensional $(\bbZ^{m+1}, \varepsilon)$-color Lie algebra with $L_{-}=0$ generated by $\{\theta_{0}, \ldots , \theta_{m} \}$ with $|\theta_{i}|=e_{i}$ for $0 \leq i \leq m$.
  Then, for $d \geq  2n_{L}-1$, a one-to-one correspondence exists between the set of isomorphism classes of truncated point modules of length $d+1$ over $U(L)$ and that of $S_{\varepsilon}(L_{\bfone})$.
    In particular, for $m \geq 2$, the set of isomorphism classes of point modules over $U(L)$ corresponds one-to-one to
  \[
    \{(p_{0}:p_{1}:\cdots:p_{m})\in \bbP^{m}\mid (\omega_{i,j}\omega_{j,l}-\omega_{i,l})p_{i}p_{j}p_{l}=0 \text{ for } 0\leq i < j < l \leq m \}
  \]
  where $\omega_{i,j}:=\varepsilon(e_{i}, e_{j})$ for $0 \leq i,j \leq m$.
  For $m=1$, the set of isomorphism classes of point modules over $U(L)$ corresponds one-to-one to the set of closed points of $\bbP^{1}$.
\end{thm}
\begin{proof}
When $n_{L}=1$, since $L=L_{\bfone}$, the argument is clear.
We use induction on $\dimk_{k}L$.
If $\dimk_{k}L=2$ holds, then $n_{L}=1$ holds.
Suppose that $\dimk_{k}L \geq 3$ and $n_{L} \geq 2$.
  Let $g=[x,y]\neq 0$ be an element of $L_{\bfone}^{n_{L}}$.
  By reselecting $g$ if necessary, one can assume that $g=[x,y] \neq 0$ satisfies $x \in h(L_{\bfone})$ and $y \in h(L)$.
  Since $L$ is generated by $L_{\bfone}$, $[g,L]=0$ holds.
  Then, $g$ is a normal element of $U(L)$.
  Also, $g$ is regular since $U(L)$ is a domain (\cite{Price05}).
  Furthermore, the equalities
  $\varepsilon(|x|,|x|+|y|)=\varepsilon(|x|,|x|)\varepsilon(|x|,|y|)=\varepsilon(|x|,|y|)$ 
  and
  $\varepsilon(|x|+|y|,|y|)=\varepsilon(|x|,|y|)\varepsilon(|y|,|y|)=\varepsilon(|x|,|y|)$ hold.
  It follows that
  \[
  \begin{cases}
    g=[x,y]=xy-\varepsilon(|x|,|y|)yx;\\
    0=[x,g]=xg-\varepsilon(|x|,|x|+|y|)gx=xg-\varepsilon(|x|,|y|)gx;\\
    0=[g,y]=gy-\varepsilon(|x|+|y|,|y|)yg=gy-\varepsilon(|x|,|y|)yg.\\
  \end{cases}
  \]
  Hence, $g$ is a $q'$-Heisenberg normal element of $U(L)$ with $\deg g = n_{L} \geq 2$.
Therefore, for $d \geq 2n_{L}-1$, the set of isomorphism classes of truncated point modules of length $d+1$ over $U(L)$ corresponds one-to-one to the one of $U(L)/\langle g \rangle$ by \propref{prop-torfree-gP} and \thmref{thm-tru-pmod}.
By a similar proof of \cite[2.2.14]{Dix}, we have $U(L)/\langle g \rangle \cong U(L/\langle g \rangle)$.
Also, clearly, $d \geq 2n_{L}-1 \geq 2n_{L/\langle g \rangle}-1$ 
and $S_{\varepsilon}(L_{\bfone}) \cong S_{\varepsilon}((L/\langle g \rangle)_{\bfone})$.
By the induction hypothesis, we have a conclusion.
The second claim follows from \thmref{thm-BDL}.
\end{proof}
Note that the second claim of \thmref{thm-color-var} also follows from \propref{prop-fat-point}.
\begin{cor}\label{cor-color-sch}
  Let $L$ be a finite dimensional $(\bbZ^{m+1}, \varepsilon)$-color Lie algebra with $L_{-}=0$ and $n_{L} \geq 2$ generated by $\{\theta_{0}, \ldots , \theta_{m} \}$ with $|\theta_{i}|=e_{i}$ for $0 \leq i \leq m$.
  Suppose that $\Gamma_{j}(U(L))$ is the $j$-th truncated point scheme of $U(L)$ and $\Gamma(U(L))$ is the point scheme of $U(L)$.
  Then, $\mathrm{pr}_{1,2n_{L}-2}^{(2n_{L}-1)}\mathrm{:}\, \Gamma_{2n_{L}-1}(U(L))\rightarrow \Gamma_{2n_{L}-2}(U(L))$ is a closed immersion and identifies $\Gamma_{2n_{L}-1}(U(L))$ with a closed subscheme $E \subset \Gamma_{2n_{L}-2}(U(L))$.
  Moreover, $\Gamma(U(L)) \cong E$ as schemes.
\end{cor}
\begin{proof}
  For a scheme $X$, the set of closed points of $X$ is denoted by $\Cl(X)$. 
  By \thmref{thm-color-var}, we obtain the equality $\Cl(\Gamma_{2n_{L}-1}(U(L)))=\Cl(\Gamma_{2n_{L}-1}(S_{\varepsilon}(L_{\bfone})))$.
  Also, for a normal element $g$ of $U(L)$ with $g \in h(L^{n_{L}}_{\bfone})$, Propositions \ref{prop-torfree-gP} and \ref{prop-ATV-corres} give $\Cl(\Gamma_{2n_{L}-2}(U(L/\langle g \rangle))) \subset \Cl(\Gamma_{2n_{L}-2}(U(L)))$.
  Inductively, we obtain 
  $\Cl(\Gamma_{2n_{L}-2}(S_{\varepsilon}(L_{\bfone}))) \subset \Cl(\Gamma_{2n_{L}-2}(U(L)))$.
  Then, by \cite[11.2.Q.Exercise]{Vak} and \cite{BDL} (also, see \cite{Vit}), $\mathrm{pr}_{1,2n_{L}-2}^{(2n_{L}-1)}$
  is injective on points.
  By \cite[Proposition 3.6 (ii)]{ATV90}, $\mathrm{pr}_{1,2n_{L}-2}^{(2n_{L}-1)}$ is a closed immersion and identifies $\Gamma_{2n_{L}-1}(U(L))$ with a closed subscheme $E \subset \Gamma_{2n_{L}-2}(U(L))$.
  Also, $\sigma: E \rightarrow \Gamma_{2n_{L}-2}(U(L))$ is defined as $(p_{1},\ldots ,p_{2n_{L}-2}) \mapsto (p_{2},\ldots ,p_{2n_{L}-1})$, where $(p_{1},\ldots ,p_{2n_{L}-1}) \in \Gamma_{2n_{L}-1}(U(L))$ is the point laying over $(p_{1},\ldots ,p_{2n_{L}-2})$ by \cite[Proposition 3.7 (i)]{ATV90}.
  By a similar discussion to that above, $\sigma(E) \subset E$ is obtained.

  Let $I$ be a two-sided ideal of $k\langle x_{0}, \ldots x_{m} \rangle$ such that $U(L) \cong k\langle x_{0}, \ldots x_{m} \rangle/I$, where $\deg x_{j}=1$ for $0 \leq j \leq m$ 
  and $B=\{b_{1}, \ldots , b_{\dimk_{k}L}\}$, a homogeneous $k$-vector basis of $L$.
  We denote the subset $B_{2}$ of $B\times B$ defined as $B_{2}=\{(b_{\ell},b_{\ell'})\in B\times B\mid 1 \leq \ell < \ell' \leq \dimk_{k}L \}$.
  Suppose that $C_{\bullet}=(C_{i},d_{i})$ is the color Koszul resolution of $U(L)$ in \thmref{thm-color-Kos}
  and $P_{\bullet}=(P_{i}, \partial_{i})$ is a minimal free resolution of $k$ as a $\bbZ$-graded $U(L)$-module.
  Then, the following chain maps exist where $f_{\bullet}: P_{\bullet} \rightarrow C_{\bullet}$
  and $\pi: C_{\bullet} \rightarrow P_{\bullet}$ such that $\pi f = \mathrm{id}$.
  Obviously, $P_{2}$ is the summand of 
  $C_{2} \cong \bigoplus_{(b,b') \in B_{2}}U(L)(-|b|_{\bbZ} - |b'|_{\bbZ})$.
  Assume that $P_{2}$ has a direct sum decomposition $P_{2}=Q \oplus Q'$ such that $Q \cong U(L)(-2n_{L})$.
  Let $w$ be a free basis of $Q$ and 
  $f_{2}(w)=\dsum_{(b,b')\in B_{2}} v_{(b,b')}\otimes (b \wedge b')$ be the image of $w$, where $v_{(b,b')} \in h(U(L))$ with $\deg v_{b,b'} + |b|_{\bbZ} + |b'|_{\bbZ} =2n_{L}$.
  Suppose that $z$ is an element of $C_{2}$ defined as $z=\dsum_{ |b|_{\bbZ} + |b'|_{\bbZ} < 2n_{L}} v_{(b,b')}\otimes (b \wedge b')$.
  If $|b|_{\bbZ} + |b'|_{\bbZ} < 2n_{L}$, then $\pi_{2}(1\otimes (b \wedge b')) \in Q'$.
  It follows that $\pi_{2}(z) \in Q'$.
  Meanwhile, if $|b|_{\bbZ} + |b'|_{\bbZ} = 2n_{L}$, then $|b|_{\bbZ} = |b'|_{\bbZ} =n_{L}$.
  Since $U(L)$ is generated by degree $1$ and $n_{L} \geq 2$, the following equality holds\textup{:}
  \begin{align*}
  \partial_{2}\pi_{2}(v_{(b,b')}\otimes (b \wedge b'))
  &=\pi_{1}d_{2}(v_{(b,b')}\otimes (b \wedge b'))
  \\
  &=\pi_{1}\left(v_{(b,b')} b \otimes b'
  -\varepsilon(|b|,|b'|) v_{(b,b')} b' \otimes b \right)
  =0.
  \end{align*}
  It follows that $\partial_{2}(w)=\partial_{2}\pi_{2}f_{2}(w)=\partial_{2}\pi_{2}(z)$, and then $\partial_{2}(P_{2})\subset \partial_{2}(Q')$ is obtained.
  This contradicts the minimality of $P_{\bullet}$.
  Hence, $P_{2}$ is isomorphic to a direct summand of 
  \[
  \Dbigoplus_{|b|_{\bbZ} + |b'|_{\bbZ} < 2n_{L},\, (b,b') \in B_{2}}U(L)(-|b|_{\bbZ} - |b'|_{\bbZ}).
  \]
  This means that $I$ is generated by relations with degree $\leq 2n_{L} -1$ or less by \cite[Lemma 2.1.3]{Rog}. 
  Therefore, by \cite[Proposition 3.7 (ii)]{ATV90}, for $r \geq 2n_{L}$, $\mathrm{pr}_{1,2n_{L}-1}^{(r)}: \Gamma_{r}(U(L)) \rightarrow E$ is an isomorphism and 
  \[
  \Gamma(U(L)) = \displaystyle{\lim_{\leftarrow}} \Gamma_{j}(U(L)) \cong E .
  \]
\end{proof}
\section*{Acknowledgments}
I am deeply grateful to my Ph.D. supervisor, Professor Ayako Itaba, for her invaluable guidance and useful discussions.
I would like to thank Shota Inoue for his helpful comments on this research.
\bibliographystyle{plain}
\bibliography{references.bib}
\end{document}